\documentclass[11pt]{amsart}

\usepackage{amssymb,amsthm,amsmath,amsxtra}
\usepackage[all]{xy}
\usepackage{fullpage}
\usepackage{hyperref}
\usepackage{comment}
\usepackage{mathrsfs}
\usepackage{colonequals}
\hypersetup{colorlinks=true,urlcolor=blue,citecolor=blue,linkcolor=blue}

\numberwithin{equation}{subsection}

\theoremstyle{plain}
\newtheorem{thm}[equation]{Theorem}
\newtheorem{prop}[equation]{Proposition}
\newtheorem{lem}[equation]{Lemma}

\newtheorem{conj}[equation]{Conjecture}

\usepackage[OT2,T1]{fontenc}
\DeclareSymbolFont{cyrletters}{OT2}{wncyr}{m}{n}
\DeclareMathSymbol{\Sha}{\mathalpha}{cyrletters}{"58}

\theoremstyle{remark}

\newtheorem{remark}[equation]{Remark}
\newtheorem{rmk}[equation]{Remark}

\DeclareMathOperator{\Gal}{Gal}

\DeclareMathOperator{\MAut}{MAut}
\DeclareMathOperator{\rk}{rk}
\DeclareMathOperator{\hht}{ht}
\DeclareMathOperator{\sgn}{sign}
\DeclareMathOperator{\Tr}{Tr}

\DeclareMathOperator{\PGL}{PGL}
\DeclareMathOperator{\SL}{SL}
\DeclareMathOperator{\Nm}{Nm}
\DeclareMathOperator{\Frac}{Frac}
\DeclareMathOperator{\End}{End}

\newcommand{\C}{\mathbb C}
\newcommand{\F}{\mathbb F}

\newcommand{\PP}{\mathbb P}
\newcommand{\Q}{\mathbb Q}
\newcommand{\Qbar}{\overline{\mathbb Q}}

\newcommand{\Z}{\mathbb Z}

\newcommand{\frakp}{\mathfrak{p}}
\newcommand{\frakH}{\mathfrak{H}}

\newcommand{\la}{\langle}
\newcommand{\ra}{\rangle}

\newcommand{\scrO}{\mathscr{O}}

\newcommand{\pdot}{j}

\newcommand{\psmod}[1]{~(\textup{\text{mod}}~{#1})}

\begin{document}

\title{Sylvester's problem and mock Heegner points}
\author{Samit Dasgupta and John Voight}
\date{\today}


\begin{abstract}
We prove that if $p \equiv 4,7 \pmod{9}$ is prime and $3$ is not a cube modulo $p$, then both of the equations $x^3+y^3=p$ and $x^3+y^3=p^2$ have a solution with $x,y \in \Q$.
\end{abstract}

\maketitle

\section{Introduction}

\subsection{Motivation}

We begin with the classical Diophantine question: which integers $n$ can be written as the sum of two cubes of rational numbers?  More precisely, let $n \in \Z_{>0}$ be cubefree, and let $E_n$ denote the projective plane curve defined by the equation $x^3+y^3=nz^3$.  Equipped with the point $\infty=(1:-1:0)$, the curve $E_n$ has the structure of an elliptic curve over $\Q$.  (The equation for $E_n$ can be transformed via a change of variables to yield the Weierstrass equation $y^2=x^3-432n^2$.)  We have $E_1(\Q) \simeq \Z/3\Z$ generated by $(1:0:1)$ and $E_2(\Q) \simeq \Z/2\Z$ generated by $(1:1:1)$; otherwise, $E_n(\Q)_{\text{tors}}=\{\infty\}$ for $n \geq 3$.  So our question becomes: for which cubefree integers $n \geq 3$ is $\rk E_n(\Q)>0$?

A conjecture, attributed to Sylvester, suggests an answer to this question when $n=p$ is prime.

\begin{conj}[Sylvester \cite{Sylvester}, Selmer \cite{Selmer}] \label{conj:sylvester}
If $p \equiv 4, 7, 8 \pmod{9}$, then $\rk E_p(\Q)>0$.
\end{conj}

An explicit 3-descent \cite{Satge} shows that
\begin{equation} \label{e:upper}
\rk E_p(\Q) \le 
\begin{cases}
0, & \text{if $p \equiv 2, 5\phantom{, 9} \pmod{9}$;}  \\
1, & \text{if $p \equiv 4, 7, 8 \pmod{9}$;}  \\
2, & \text{if $p \equiv 1\phantom{, 9, 9} \pmod{9}$.} 
\end{cases} 
\end{equation}
In particular, primes $p \equiv 2,5 \pmod{9}$ are \emph{not} the sum of two cubes, a statement that can be traced back to P\'epin, Lucas, and Sylvester \cite[Section 2, Title 1]{Sylvester}.  

At the same time, the sign of the functional equation for the $L$-series of $E_p$ is
\begin{equation} \label{e:sign} \sgn(L(E_p/\Q, s)) = 
\begin{cases} -1, & \text{if $p\equiv 4,7,8 \pmod{9}$;}  \\  
+1, & \text{otherwise.} \\ 
\end{cases}  
\end{equation}
Putting these together, for $p \equiv 1 \pmod{9}$, the Birch--Swinnerton-Dyer (BSD) conjecture predicts that $\rk E_p(\Q)=0\text{ or }2$, depending on $p$ in a nontrivial way.  This case was investigated by Rodriguez-Villegas and Zagier \cite{RVZ}: they give three methods to determine for a given prime $p$ whether or not $\rk E_p(\Q)=0$.

\subsection{Main result}

We are left to consider the cases $p \equiv 4,7,8 \pmod{9}$.  The BSD conjecture together with (\ref{e:upper}) and (\ref{e:sign}) then predicts that $\rk E_p(\Q)=1$, and hence that $p$ is the sum of two cubes.  In this article, we prove the following (unconditional) result as progress towards Sylvester's conjecture.

\begin{thm} \label{t:main}
Let $p \equiv 4,7\pmod{9}$ be prime and suppose that $3$ is not a cube modulo $p$.  Then $\rk E_p(\Q)=\rk E_{p^2}(\Q)=1$.
\end{thm}

In 1994, Elkies \cite{Elkies} announced a proof of the stronger statement that the conclusion of Theorem~\ref{t:main} holds for \emph{all} $p \equiv 4,7 \pmod{9}$.  The details of the proof have not been published, but his methods differ substantially from ours \cite{ElkiesWWW}.

Theorem \ref{t:main} was announced and the proof sketched in earlier work \cite{DV}, but several important details were not provided and are finally given here.  The construction in this paper has been recently used by Shu--Yin \cite{ShuYin} to prove that the power of $3$ dividing $\#\Sha(E_p)\#\Sha(E_{3p^2})$ is as predicted by the BSD conjecture, following a method similar to the work of Cai--Shu--Tian \cite{CST15}.  (See also section 1.4 below.)
We are not aware of any results concerning the case $p \equiv 8 \pmod{9}$ of Conjecture \ref{conj:sylvester}, which appears to be decidedly more difficult.

\subsection{Sketch of the proof}

We now discuss the proof of Theorem~\ref{t:main}.  General philosophy predicts that in the situation where the curve $E_p$ has expected rank $1$, one should be able to construct rational nontorsion points on $E_p$ using the theory of complex multiplication (CM).  One might first consider the classical method of Heegner points.  We start with the modular parameterization $\Phi:X_0(N) \to E_p$, where $N$ is the conductor of $E_p$, given by
\[ N=\begin{cases} 27p^2, & \text{if $p \equiv 4\pmod{9}$,} \\
9p^2, & \text{if $p \equiv 7\pmod{9}$.} 
\end{cases} \]
Given a quadratic imaginary field $K$ that satisfies the \emph{Heegner hypothesis} that both $3$ and $p$ are split, we may define a cyclic $N$-isogeny that yields a point $P \in X_0(N)(H)$, where $H$ denotes the Hilbert class field of $K$.  The trace $Y=\Tr_{H/K} \Phi(P)$ yields a point on $E_p(K)$.  By the Gross-Zagier formula \cite{GZ}, we expect this point to be nontorsion.  Indeed, the BSD conjecture (which in particular furnishes an equality of the algebraic and analytic ranks of $E_p$) implies that this is the case.  But in order to apply this method, we must first choose a suitable imaginary quadratic field $K$, and no natural candidate for $K$ presents itself; after making such a choice, it is unclear how to prove unconditionally that the resulting Heegner points are nontorsion.  

Instead, in this article we work with what are known as \emph{mock Heegner points}.  This terminology is due to Monsky \cite[p.~46]{Monsky}, although arguably Heegner's original construction may be described as an example of such ``mock'' Heegner points.  We consider the field $K = \Q(\sqrt{-3}) = \Q(\omega)$, where $\omega = \exp(2\pi i /3)$ is a primitive cube root of unity.  Note that the elliptic curve $E_n$ has CM by the ring of integers $\Z_K = \Z[\omega]$, and that the prime $3$ is ramified in $K$, so the Heegner hypothesis is not satisfied.  Nevertheless, Heegner-like constructions of points defined by CM theory may still produce nontorsion points in certain situations: for example, one can show that results of Satg\'e  \cite{Satge} concerning the curve $x^3+y^3=2p$ can be described in the framework of mock Heegner points \cite{DV}.

We take instead a fixed modular parametrization $X_0(243) \to E_9$.  We  consider an explicit cyclic $243$-isogeny of conductor $9p$ which under this parameterization yields a point $P \in E_9(H_{9p})$, where $H_{9p}$ denotes the ring class field of $K$ associated to the conductor $9p$.  We descend the point $P \in E_9(H_{9p})$ with a twist by $\sqrt[3]{3}$ to a point $Q \in E_1(H_{3p})$.   This descent argument is particularly appealing and non-standard because it compares the action of the exotic modular automorphism group of $X_0(243)$ as studied by Ogg \cite{Ogg} to the Galois action on CM points provided by the Shimura Reciprocity Law.

We next consider the trace $R=\Tr_{H_{3p}/L} Q \in E_1(L)$, where $L=K(\sqrt[3]{p})$.   We  show that after translating by an explicit torsion point, $R$ twists to yield a point $Z \in E_p(K)$ or $Z \in E_{p^2}(K)$, depending on the original choice of $243$-isogeny.  Again this argument employs the group of exotic modular automorphisms of $X_0(243)$.

We conclude by showing that the point $R$ (hence $Z$) is nontorsion when 3 is not a cube modulo $p$, and this implies the theorem since $\rk_{\Z} E_n(\Q) = \rk_{\Z_K} E_n(K)$.
To do this we  consider the reduction of $R$ modulo the primes above $p$.  By an explicit computation with $\eta$-products, we show that when $3$ is not a cube modulo $p$, this reduction is not the image of any torsion point in $E_1(L)$: see Proposition \ref{prop:3notpmainprop}.  This reduction uses in a crucial way a generalization and refinement of Kronecker's congruence: see Proposition \ref{shim2}.  In the end, we are able to show that when 3 is not a cube modulo $p$, the point $R$ is nontorsion because it is not congruent to any torsion point modulo $p$.  Without the descent made possible by the exotic modular automorphism group of $X_0(243)$, our point $Z$ (e.g., which could have been defined more simply by taking an appropriate ``twisted'' trace of $P$  from $H_{9p}$ to $K$) would have been twice multiplied by 3, and the delicate proof that it is nontorsion would have fallen through.

\subsection{Heuristics and the work of Shu and Yin} 

We now explain why it should be expected that the condition ``3 is not a cube modulo $p$'' should appear in the statement of Theorem~\ref{t:main} for our construction.  As mentioned above, our setting does not satisfy the Heegner hypothesis and hence the classical Gross--Zagier formula does not apply in this case.  Nevertheless, Shu and Yin have proven the following result.

\begin{thm}[{\cite[Theorem 4.4]{ShuYin}}] \label{t:sy}  
Let $p \equiv 4,7 \pmod{9}$ be prime.  Let $\chi_{3p}:\Gal(H_{3p} \,|\, K) \to \mu_3$ be the cubic character associated to the field 
$K(\sqrt[3]{3p})$, i.e., \[ \chi(\sigma) = \sigma\bigl(\sqrt[3]{3p}\bigr)/\sqrt[3]{3p} \quad \text{for $\sigma \in \Gal(H_{3p}\,|\, K)$.} \]
Let $Z \in E_p(K)$ be the mock Heegner point constructed above.  Then 
\[ \frac{ L'(E_9/K, \chi_{3p}, 1) }{\Omega} =  c \cdot \hht(Z), \]
where the complex period $\Omega \in \C^\times$ and rational factor $c \in \Q^\times$ are explicitly given.
\end{thm}

The Artin formalism  for $L$-functions yields
\[ L(E_9/K, \chi_{3p}, s) = L(E_p/\Q, s)L(E_{3p^2}/\Q, s), \]
and hence Theorem~\ref{t:sy} relates $\hht(Z)$ to \begin{equation} \label{e:factored} L'(E_p/\Q, 1)L(E_{3p^2}/\Q, 1). \end{equation}
Therefore we should expect that $Z$ is nontorsion if and only if $L(E_{3p^2}/\Q, 1) \neq 0$.
In fact, it is possible to have $L(E_{3p^2}/\Q, 1) = 0$ (e.g., $p = 61,193$), and in such cases our point $Z \in E_p(K)$ is torsion.

However, whenever $3$ is not a cube modulo $p$, one can show that the Selmer group associated to a certain rational 3-isogeny to $E_{3p^2}$ is trivial
(see Satg\'e \cite[Theorem 2.9(3) and p.~313]{Satge}) and consequently that $E_{3p^2}(\Q)$ is finite and hence by BSD that $L(E_{3p^2}/\Q, 1) \neq 0$.  This explains why it is reasonable to expect this condition to appear in the statement of Theorem~\ref{t:main}.  The appeal of Theorem~\ref{t:main} is that it is explicit and unconditional---i.e., it does not depend on BSD, even though BSD and the theorem of Shu--Yin explain {\em why} the condition on $3$ modulo $p$ should be expected to  appear in the statement.

\subsection{Organization}

In  \S\ref{sec:setup} we  describe our explicit modular parameterization and the group of modular automorphims of $X_0(243)$. In  \S\ref{sec:isogtreedesc} we define our mock Heegner points, and in \S\ref{sec:descenttracing} we descend and trace them to define points over $K$.  In \S\ref{sec:nontors}, we prove that our points are nontorsion when 3 is not a cube modulo $p$.

\subsection{Acknowledgements}

The authors wish to thank Brian Conrad, Henri Darmon, and Noam Elkies for helpful discussions as well as the hospitality of the Centre Recherche de Math\'ematiques (CRM) in Montr\'eal where part of this work was undertaken in December 2005.  The authors would also like to thank Hongbo Yin for comments and the encouragement to publish this work.

\section{The modular curve $X_0(243)$} \label{sec:setup}

Throughout, let $K \colonequals \Q(\omega) \subset \C$ where $\omega \colonequals (-1+\sqrt{-3})/2$ is in the upper half-plane and $\Z_K \colonequals \Z[\omega]$ its ring of integers.  
We begin in this section by setting up a few facts about the modular curve $X_0(243)$.  

\subsection{Basic facts}

The (smooth, projective, geometrically integral) curve $X_0(243)$ over $\Q$ is the coarse moduli space for cyclic $243$-isogenies between (generalized) elliptic curves, and there is an isomorphism of Riemann surfaces
\[ X_0(243)(\C) \xrightarrow{\sim} \Gamma_0(243) \backslash \frakH^* \]
where $\frakH^* \colonequals \frakH \cup \PP^1(\Q)$ is the completed upper half-plane.  Explicitly, to $\tau \in \frakH$ we associate the cyclic isogeny 
\begin{equation} \label{eqn:normalizedisog}
\begin{aligned}
\phi_\tau\colon \C/\langle 1,\tau \rangle &\to \C/\langle 1,243\tau \rangle \\
z &\mapsto 243 z
\end{aligned}
\end{equation}
with $\ker \phi_\tau$ generated by $1/243$ in the lattice $\Z + \Z\tau$.  The genus of $X_0(243)$ is $19$.

For further reading on automorphism groups of modular curves, we refer to Ogg \cite{Ogg}.  The group of modular automorphisms of $X_0(243)$ is by definition 
\[ \MAut(X_0(243)) \colonequals N_{\PGL_2^+(\Q)}(\Gamma_0(243))/\Gamma_0(243) \]
where $N$ denotes the normalizer.  The group $\MAut(X_0(243))$ is generated by an \emph{exotic} automorphism $v \colonequals \begin{pmatrix} 1 & 0 \\ 81 & 1 \end{pmatrix} \in \MAut(X_0(243))$ of order $3$ and the Atkin--Lehner involution $w \colonequals \begin{pmatrix} 0 & -1 \\ 243 & 0 \end{pmatrix} \in \MAut(X_0(243))$ of order $2$.  We find
\[ \MAut(X_0(243)) = \langle w,v^{-1}wv \rangle \times \langle v \rangle \simeq S_3 \times \Z/3\Z. \]
The subgroup of $\MAut(X_0(243))$ isomorphic to $S_3$ is characteristic, and we let $\Gamma \leq \PGL_2^+(\Q)$ be the subgroup generated by $\Gamma_0(243)$ and $S_3$.  One can check that $v$ normalizes $\Gamma$.
Moreover, the matrix $t \colonequals \begin{pmatrix} 9 & 1 \\ -243 & -18 \end{pmatrix}$ normalizes the group $\MAut(X_0(243))$ and the group $\Gamma$.  (But $t$ does not normalize $\Gamma_0(243)$ itself.)  One can  check that $t^3 = 729$ is scalar, so $t$ has order $3$ as a linear fractional transformation.  

\subsection{Explicit modular parametrization and modular automorphisms} \label{s:modparam}

We now consider the quotient of $X_0(243)$ by the subgroup $S_3 < \MAut(X_0(243))$
\begin{equation} \label{e:qm}
 X_0(243) \to X_0(243)/S_3 = X(\Gamma) \end{equation}
where $X(\Gamma) \colonequals \Gamma \backslash \frakH^*$.  Riemann--Hurwitz shows  that the genus of $X(\Gamma)$ is $1$, and the image of the cusp $\infty \in X_0(243)(\Q)$ gives it the structure of an elliptic curve over $\Q$.  This quotient morphism (\ref{e:qm}) is defined over $\Q$ and has a particularly pleasing realization as follows.  Let 
\begin{equation} 
\eta(z) \colonequals q^{1/24}\prod_{n=1}^{\infty}(1-q^n)
\end{equation}
with $q \colonequals \exp(2\pi iz)$ be the Dedekind $\eta$-function.

\begin{prop} \label{modularparam}
We have a modular parametrization
\begin{align*}
\Phi\colon X_0(243) \to X(\Gamma) &\xrightarrow{\sim} E_9\colon y^2+y=x^3-1 \\
z &\mapsto (x,y)
\end{align*}
where
\begin{equation} \label{eqn:xyz}
x(z)=\displaystyle{\frac{\eta(9z)\eta(27z)}{\eta(3z)\eta(81z)}}, \quad y(z)= - \displaystyle{\frac{\eta(9z)^4+9\eta(9z)\eta(81z)^3}{\eta(27z)^4-3\eta(9z)\eta(81z)^3}} - 2. 
\end{equation}
\end{prop}

\begin{proof}
The $\eta$-product $x(z)$ is a modular function on $X_0(243)$ by Ligozat's criterion \cite[Theorem 2]{kilford}.  By the transformation properties of the $\eta$-function, it is straightforward to show that $x(z)$ is invariant under the action of the subgroup $S_3 < \MAut(X_0(243))$.  

The function $y(z)$ was discovered on a computer experimentally by manipulating $\eta$-products via their $q$-expansions.  In a similar way, one can show that $y$ is invariant under $\Gamma_0(243)$ and the subgroup $S_3 < \MAut(X_0(243))$.  To prove that the equality $y^2 + y = x^3 - 1$ holds, after clearing denominators we may equivalently show an equality of holomorphic modular forms of weight 7---but then it suffices to verify the equality on enough terms of the $q$-expansions on a computer to satisfy the Hecke bound.
\end{proof}

\begin{remark}
The elliptic curve $E_9$ of conductor $243$ is number \textsf{243a1} in the tables of Cremona and has LMFDB label 
\href{http://www.lmfdb.org/EllipticCurve/Q/243/a/1}{\textsf{243.a1}}.  
\end{remark}

\begin{remark}
One can show that the $y$-function in \eqref{eqn:xyz} cannot be expressed simply as an $\eta$-product, moreover there is no $\eta$-product that is invariant under $S_3$ and has a pole of order $3$ at the preimage of the origin  in  $E_9$. We do not use the explicit formula for $y(z)$ in this paper.
\end{remark}

Because the matrices $t,v$ normalize $\Gamma$, they give rise to automorphisms of $E_9$ as a genus $1$ curve.  
The endomorphism ring of $E_9$ as an elliptic curve is $\Z_K = \Z[\omega]$, where $\omega$ acts via $(x, y) \mapsto (\omega x, y)$.  Every endomorphism of $E_9$ as a genus 1 curve has the form $Z \mapsto aZ + b$ where $a \in \Z_K$ and $b \in E_9$.
The following proposition describes the automorphisms $t$ and $v$ of $E_9$ explicitly in these terms.

\begin{prop}  \label{p:vtact} 
The automorphism $t$ acts on the curve $E_9$ via $t(Z)=\omega^2 Z+(0,\omega)$.
The automorphism $v$ acts by $v(Z) = \omega^2(Z)$.
\end{prop}

\begin{proof}
Since $t^3$ is a scalar matrix, $t(Z)=aZ+b$ for $a \in \{1,\omega,\omega^2\}$ and $b \in E_9(\Qbar)$.  
Now $t(\infty)=-1/27$ and under the complex parametrization $\Phi$ we compute that $\Phi(-1/27)=(0,\omega)=b$.  Unfortunately, we cannot determine $a$ by looking at cusps.  Instead, we consider $\tau=(\omega-1)/27 \in \frakH$, which has the property that $t(\tau)=\tau$.  Letting $T=\Phi(\tau)$, it follows that  $(1-a)T=b=(0,\omega)$. In particular  $T \in E_9[3]$. We compute numerically that $T \approx (\sqrt[3]{3},-2)$, and since there are only 9 possibilities for $T$, equality holds.  From this, one finds that $a=\omega^2$, and hence $t(Z)=\omega^2 Z+(0,\omega)$.

Next we compute the action of $v$, which also has order dividing $3$, so again $v(Z)=aZ+b$ with $a \in \{1,\omega,\omega^2\}$.  We see that $v(\infty)=1/81$ and $\Phi(1/81)=\infty$ so $b=0$.  As above, we compute that $\tau=(\omega-1)/27$ has $\Phi(\tau)=T=(\sqrt[3]{3},-2)$ is a $3$-torsion point, hence $\Phi(v(\tau))=a(\sqrt[3]{3},-2)$ is also a $3$-torsion point and then we verify numerically that $a=\omega^2$.  
\end{proof}

\section{Mock Heegner points} \label{sec:isogtreedesc}

For the remainder of this paper, let $p$ be a prime congruent to 4 or 7 modulo 9.  In this section, we define our mock Heegner point.

\subsection{The isogeny tree}  

In Figure~\ref{f:isogenytree} below, we draw a diagram of 3-isogenies between certain elliptic curves with CM by orders in $K$.  For $\tau~\in~K~\cap~\frakH$, we denote by $\langle \tau \rangle_f$ the elliptic curve $\C/(\Z + \Z \tau)$ with endomorphism ring the order $\Z_{K, f} \colonequals \Z[f \omega]$ of conductor (or index) $f$ in $\Z_K$.  

\begin{figure}[h] \label{f:isogenytree}
\[ 
\xymatrix@R=0.5em@C=1em
{ \langle \frac{\omega p + 6}{9} \rangle_{9p} \ar@{-}[dr] & & & & \langle \frac{\omega p + 1}{9} \rangle_{9p}\ar@{-}[dl] \\
  \langle \frac{\omega p + 3}{9} \rangle_{9p}\ar@{-}[r]  & \langle \frac{\omega p}{3} \rangle_{3p}  \ar@{-}[dr] & & 
     \langle \frac{\omega p + 1}{3} \rangle_{3p}\ar@{-}[dl] &  \langle \frac{\omega p + 4}{9} \rangle_{9p}\ar@{-}[l]  \\
    \langle \frac{\omega p}{9} \rangle_{9p}\ar@{-}[ur] & &  \langle \omega p  \rangle_{p}  & &  \langle \frac{\omega p + 7}{9} \rangle_{9p} \ar@{-}[ul] &   \langle \frac{\omega p + 2}{27}\rangle_{9p} \ar@{-}@/_0.7pc/[dddl]  \\
    & & & & & \langle\frac{\omega p + 11}{27}\rangle_{9p}  \ar@{-}@/_0.2pc/[ddl] \\
        & & & & & \langle\frac{\omega p + 18}{27}\rangle_{9p}  \ar@{-}[dl] \\
    \langle \frac{3\omega p + 1}{3} \rangle_{9p} \ar@{-}[ddr]  & & & &   \langle \frac{\omega p + 2}{9}\rangle_{3p} \ar@{-}[ddl] 
  &    \langle\frac{\omega p + 5}{27}\rangle_{9p}  \ar@{-}@/_0.3pc/[ddl]  \\
           & & & & & \langle\frac{\omega p + 14}{27}\rangle_{9p}  \ar@{-}[dl] \\  
  \langle \frac{3\omega p + 2}{3} \rangle_{9p}\ar@{-}[r]  & \langle 3\omega p \rangle_{3p}\ar@{-}[uuuuur] & & 
   \langle  \frac{\omega p + 2}{3} \rangle_{p}\ar@{-}[uuuuul] & \langle \frac{\omega p + 5}{9}\rangle_{3p}\ar@{-}[l] & 
    \langle\frac{\omega p + 23}{27}\rangle_{9p}  \ar@{-}[l] \\
    & & & & & \langle\frac{\omega p + 8}{27}\rangle_{9p}  \ar@{-}[dl]   \\ 
    \langle 9\omega p  \rangle_{9p}\ar@{-}[uur] & & & &   \langle \frac{\omega p + 8}{9}\rangle_{3p}\ar@{-}[uul] & 
    \langle\frac{\omega p + 17}{27}\rangle_{9p}  \ar@{-}[l] \\
          & & & & & \langle\frac{\omega p + 26}{27}\rangle_{9p}  \ar@{-}[ul]  
}
\] \\
Figure \ref{f:isogenytree}: Isogeny tree (for $p \equiv 1 \pmod{3}$)
\end{figure}

The computation of the conductors in Figure~\ref{f:isogenytree} relies only the fact that $p \equiv 1 \pmod{3}$.
Of particular interest in this diagram is the fact that the curves in the lower right quadrant emanating from the
``central vertex'' $\langle \omega p \rangle_p$ have endomorphism ring of lower conductor than their counterparts in the other quadrants.  We have only listed the 9 curves in the tree of distance 3 from this central vertex in this quadrant for space reasons, since these are the only curves that we will use. 

Each path of length 5 in this tree (with no backtracking) corresponds to a cyclic $3^5$-isogeny and hence yields a corresponding point on $X_0(243)$.  Furthermore, the conductor of the order associated to this cyclic 243-isogeny will be the least common multiple of the conductors of the orders of the two curves involved in the isogeny.  In particular, for each curve $\langle \tau \rangle_{9p} $ on the left side of this diagram and each $\langle \frac{\omega p + i }{27}\rangle_{9p}$ with $i \equiv -1 \pmod{3}$ on the right, there is a point on $X_0(243)(\C)$ of conductor $9p$ corresponding to the isogeny between these two curves.

\subsection{Our mock Heegner points}

Recall that our eventual goal is to produce rational points on the curves $E_{p}$ and $E_{p^2}$; we refer to these as \emph{case} 1 and \emph{case} 2, and we will eventually show that our points land on the curve $E_p$ or $E_{p^2}$, accordingly.  Our construction starts with the  points on $X_0(243)$ of conductor $9p$ corresponding to the following isogenies in each of these cases.  We make the following choices:
\begin{equation} \label{e:pdef}
P_0 = \begin{cases}
\displaystyle{ \Bigl\langle  \frac{\omega p}{9}  \Bigr\rangle \rightarrow \Bigl\langle \frac{\omega p + 23}{27}\Bigr\rangle = \Bigl\langle \frac{\omega p - 4}{27}\Bigr\rangle}, &
\text{ in case 1}; \\
 \displaystyle{    \Bigl\langle \frac{\omega p}{9}  \Bigr\rangle \rightarrow \Bigl\langle \frac{\omega p + 26}{27} \Bigr\rangle=  \Bigl\langle \frac{\omega p - 1}{27} \Bigr\rangle}, &  \text{ in case 2}.
\end{cases}
\end{equation}
This gives $P_0 \in X_0(243)(\C)$ and we write 
\begin{equation} \label{e:defp}
P = \Phi(P_0) \in E_{9}(\C).  
\end{equation}

\begin{rmk}
In fact, each of the $6 \cdot 9 = 54$ possible choices gives rise to a point on either $E_p$ or $E_{p^2}$ by the procedure we will outline, and we have simply made a choice.
\end{rmk}

\begin{lem} \label{lem:findmat}
The point $P_0 \in X_0(243)$ is represented in the upper half plane by  $\tau = M(\omega p/9)$ where $M=\begin{pmatrix} 2 & -1 \\ 9 & -4 \end{pmatrix}$ for case \textup{1} and $M=\begin{pmatrix} 1 & 0 \\ -9 & 1 \end{pmatrix}$ for case \textup{2}.
\end{lem}

\begin{proof}
We explain case 2, with case 1 being similar.  We need to rewrite the isogeny $P_0$ in normalized terms \eqref{eqn:normalizedisog}.  The isogeny $\phi$ is $\langle \omega p/9 \rangle \to \langle \omega p \rangle \to \langle (\omega p-1)/27 \rangle$ defined by $z \mapsto 9z$; thus, the kernel of $\phi$ is cyclic generated by $(\omega p-1)/243$ (modulo the lattice $\langle \omega p/9 \rangle$).  
We want a matrix $M = \begin{pmatrix} a & b \\ c & d \end{pmatrix} \in \SL_2(\Z)$ such that the diagram
\begin{equation}
\begin{minipage}{\textwidth}
\xymatrix@R=2em@C=4.5em{ 
\langle \omega p/9 \rangle \ar[r]^(.45){z \mapsto 9z} \ar[d]_{z \mapsto \frac{z}{c(\omega p/9)+d}} & 
\langle (\omega p-1)/27 \rangle \ar[d]^{z \mapsto \frac{27z}{c(\omega p/9)+d}} \\
\langle M(\omega p/9) \rangle \ar[r]^(.45){z \mapsto 243z} & \langle 243 M(\omega p/9) \rangle } 
\end{minipage}
\end{equation}
commutes.  The matrix $M=\begin{pmatrix} 1 & 0 \\ -9 & 1 \end{pmatrix}$ will do: indeed, a generator for the kernel of the isogeny shifts to $ \frac{(\omega p-1)/243}{-9(\omega p/9)+1} = -\frac{1}{243}.$
\end{proof}

\section{Descent and tracing}  \label{sec:descenttracing}

With our points in hand, via descent and tracing, we now show how to use the point  $P$ defined in (\ref{e:defp}) to construct points on $E_p(K)$ and  $E_{p^2}(K)$.    
 
\subsection{Field diagram}

Let $H_f \supseteq K$ be the ring class field attached to the conductor $f \in \Z_{\geq 1}$.  We have the following diagram of fields.
\begin{equation} \label{eqn:fielddiagram}
\begin{minipage}{\textwidth}
\xymatrix@R=0.25em@C=2em{
& & & & H_{9p}=H_{3p}(\sqrt[3]{3}) \ar@{-}_(.65){3}[dl] \\
& & & H_{3p} \ar@{-}_{(p-1)/3}[dl] \\
& & L=K(\sqrt[3]{p}) \ar@{-}_(.65){3}[dl] \\
& K \ar@{-}^2[dl] \\
\Q} 
\end{minipage}
\end{equation}

By the main theorem of complex multiplication, $P \in E_9(H_{9p})$.  Since $K$ has class number 1,  the Artin reciprocity map of class field theory yields  a canonical isomorphism
\begin{equation} \label{e:cft}
 \Gal(H_f\,|\,K) \simeq  (\Z_K/f\Z_K)^\times/(\Z/f\Z)^\times \Z_K^\times. 
\end{equation}

\subsection{Cubic twists}

We pause to recall the behavior of cubic twists in our context, referring to Silverman \cite[X.2]{silverman} for the general theory.
Let $K' \supseteq K$ be an algebraic extension and let $a \in (K')^\times \smallsetminus (K')^{\times 3}$, so $L' \colonequals K'(\sqrt[3]{a})$ has $[L':K'] = 3$.
Let 
\begin{equation} 
\rho \in \Gal(L'\,|\,K') \simeq \Z/3\Z 
\end{equation}
be the generator satisfying $\rho(\sqrt[3]{a}) = \omega \sqrt[3]{a}$.  Then for any $b \in (K')^\times$, there is an isomorphism of groups
between the subgroup of $E_{b}(L')$ that transforms under $\rho$ by multiplication by $\omega$ and $E_{ab}(K')$:
\begin{equation} \label{e:twist3} 
E_{b}(L')^{\rho = \omega} :=  \{ X \in E_b(L'): \rho(X) = \omega X \} \xrightarrow{\sim} E_{ab}(K').
\end{equation}

\subsection{Descent from $H_{9p}$ to $H_{3p}$}

We first apply the method of cubic twisting in the previous section to the extension $H_{9p} = H_{3p}(\sqrt[3]{3})$ over $H_{3p}$.  Let $\rho \in \Gal(H_{9p}\,|\,H_{3p})$ be the generator satisfying $\rho(\sqrt[3]{3}) = \omega \sqrt[3]{3}$.  The first step of our descent will be to show that the point $P  = \Phi(P_0) \in E_9(H_{9p})$ defined in (\ref{e:pdef}) lies in the left-hand side of (\ref{e:twist3}) and hence corresponds to a point in $E_1(H_{3p})$.
In the models 
\begin{equation} \label{e:model}
E_9 \colon y^2+y=x^3-1 \quad \text{ and } \quad E_1 \colon y^2 + y = 3x^3 - 1, 
\end{equation}
this twisting isomorphism becomes
\begin{equation} \label{r:xq}
\begin{aligned}
E_9(H_{9p})^{\rho = \omega} &\rightarrow E_1(H_{3p}) \\
(x,y) &\mapsto (x/\sqrt[3]{3}, y).
\end{aligned}
\end{equation}

\begin{prop} \label{p:rhotwist}
For the points $P \in E_9(H_{9p})$ defined in \textup{\eqref{e:defp}}, we have $\rho(P) = \omega P$.
\end{prop}

\begin{proof}
The idea of the proof is to use the Shimura reciprocity law to calculate the action of $\rho$ on $P$, and then to identify the image of this Galois action as the image of $P$ under the action of a {\em geometric} modular automorphism of $X_0(243)$. Using the computations from \S\ref{s:modparam} for the action of the group of modular transformations under the parameterization $\Phi$, we  deduce the desired result.

The field $K(\sqrt[3]{3})$ has conductor 9 over $K$.  The element $\beta = 1 + 3 \omega$ satisfies
\begin{equation} 
3^{(\Nm(\beta)-1)/3}\equiv (-1/\omega)^2 \equiv \omega \pmod{\beta}, 
\end{equation}
and hence under the isomorphism (\ref{e:cft}) with $f =9$, the element $\beta$ correponds to the automorphism of $K(\sqrt[3]{3})/K$ sending 
$\sqrt[3]{3} \mapsto \sqrt[3]{3} \omega$.  To lift this to the element $\rho \in \Gal(H_{9p}\,|\,H_{3p})$ exhibiting the same action on 
$\sqrt[3]{3}$, we must therefore find an element $\alpha_\rho$ such that  $\alpha_\rho \equiv 1 \pmod{3p}$
and $ \alpha_\rho \equiv \beta \pmod{9}$.
The element $\alpha_\rho=1+3p\omega$ suffices. 

Since the inverse of $\alpha_\rho$ in the left side of (\ref{e:cft}) for $f = 9p$ is $1 + 3p\omega^2$, the Shimura reciprocity Law \cite[Theorem 3.7]{Darmon} implies that in case 2, $\rho(P_0)$ is the point on $X_0(243)$ associated to the cyclic 243-isogeny
\begin{equation}  
I_\rho  \cdot  \left\langle  \frac{\omega p}{9}  \right \rangle \rightarrow I_\rho \cdot \left \langle \frac{\omega p - 1}{27} \right \rangle, 
\end{equation}
where 
\begin{equation} 
I_\rho \colonequals (1+3p\omega^2)\Z_K \cap \Z_{K,9p} = (9p^2-3p+1, 3+9p\omega^2) \subset \Z_{K,9p} 
\end{equation}
is an invertible ideal in the order $\Z_{K, 9p}$.
(Even before carrying out this calculation, the isogeny tree in Figure~\ref{f:isogenytree} implies that the result must be an isogeny between one of the curves $\langle \frac{\omega p + k}{9} \rangle$ with $k =0, 3, $ or $6$ and one of the curves 
$ \langle \frac{\omega p - j}{27} \rangle $ with $j = 1, 10,$ or $19$, since the adjacent curves in the tree have conductor $3p$ and are hence fixed by $\rho$.)  A simple calculation shows that the result is 
\begin{equation} 
\rho(P_0) = \left( \left \langle \frac{\omega p + 6}{9} \right \rangle \rightarrow \left \langle \frac{\omega p - 10}{27} \right\rangle \right). 
\end{equation}

We now look for a modular automorphism $A \in \MAut(X_0(243))$ such that $A(P_0)= \rho(P_0)$.  A quick computer search over the finite group $\MAut(X_0(243))$ reveals that the matrix $A=\begin{pmatrix} 327 & 2 \\ 53460 & 327 \end{pmatrix}$, corresponding to the element $(v^{-1} w v w)v^2  \in S_3 v^2 \subset \MAut(X_0(243))$, satisfies this condition.  Therefore, since the action of $S_3$ fixes the image on $E_9$ and $v$ acts by $\omega^2$ on $E_9$ by Proposition~\ref{p:vtact}, we conclude $A(P)=\rho(P)=\omega P$.   
A similar computation holds in case 1.
\end{proof}

From Proposition \ref{p:rhotwist}, it follows that each point $P \in E_9(H_{9p})$ defined in \eqref{e:defp} descends with a cubic twist by $3$ to a point $Q \in E_1(H_{3p})$.  

\subsection{Trace and descent from $H_{3p}$ to $L$} 

Recall from \eqref{eqn:fielddiagram} that $L = K(\sqrt[3]{p}) \subset H_{3p}$.  Define 
\begin{equation}  \label{eqn:pointR}
R \colonequals \Tr_{H_{3p}/L} Q \in E_1(L). 
\end{equation}
Let $\sigma$ be the generator of $\Gal(L\,|\,K)$ such that $\sigma(\sqrt[3]{p}) = \omega \sqrt[3]{p}$. 

\begin{prop}   \label{p:sigmar} Using the model $y^2 + y = 3x^3 - 1$ for $E_1$ as in \textup{\eqref{e:model}}:
\[  \sigma(R) = \begin{cases} \omega R + (0, \omega^2), & \textup{ in case 1;} \\
\omega^2 R + (0, \omega^2), & \textup{ in case 2}.
\end{cases}
\]
\end{prop}

\begin{proof} The proof is similar to that of Proposition~\ref{p:rhotwist} so we only sketch the salient points.
The element  $\alpha_\sigma=1-2p\omega^2 \in \Z_K$ has the property that under the Artin reciprocity isomorphism (\ref{e:cft}) for $f = 9p$, the associated element $\sigma \in \Gal(H_{9p}\,|\,K)$ satisfies $\sigma(\sqrt[3]{p}) = \omega \sqrt[3]{p}$
and $\sigma(\sqrt[3]{3}) =  \sqrt[3]{3}$.  This latter fact will be important to ensure that the $\sqrt[3]{3}$ twisting isomorphism (\ref{e:twist3}) is equivariant for the action of $\sigma$.

The Shimura reciprocity law yields the action of $\sigma$ on $P_0$, calculated using $\alpha_\sigma$ as in the proof of Proposition~\ref{p:rhotwist}.  Here one must further consider the cases $p \equiv 4, 7 \pmod{9}$ separately. 
One obtains:
\begin{equation}
\sigma(P_0) = 
\begin{cases}
\langle \frac{\omega p + 4}{9}  \rangle \rightarrow \langle  \frac{\omega p +2 }{27} \rangle, & \text{in case 1 with $p \equiv 4 \psmod{9}$;} \\
\langle \frac{\omega p + 4}{9}  \rangle \rightarrow \langle  \frac{\omega p -13}{27} \rangle, & \text{in case 2 with $p \equiv 4 \psmod{9}$;} \\
\langle 9 \omega p  \rangle \rightarrow \langle  \frac{\omega p -1 }{27}\rangle,  & \text{in case 1 with $p \equiv 7 \psmod{9}$;} \\
\langle 9\omega p   \rangle \rightarrow \langle  \frac{\omega p +2}{27} \rangle, & \text{in case 2 with $p \equiv 7 \psmod{9}$.}
\end{cases}
\end{equation}

In each case, we can again identify a modular automorphism that sends $P_0$ to $\sigma(P_0)$.  For example, in case 2 for $p \equiv 4 \pmod{9}$, we find that the matrix $A=\begin{pmatrix} 18486 & 103 \\ 27459 & 153 \end{pmatrix}$, corresponding to the element $(wvwv^2)t^2v^2$, has $A(P_0)=\sigma(P_0)$.  Since $wvwv^2 \in S_3$, we conclude using Proposition~\ref{p:vtact} that 
\begin{equation} 
\sigma(P)=A(P)=\omega^2 P+(0,\omega^2). 
\end{equation}
The results in all 4 cases are:
\begin{equation}
\sigma(P) = 
\begin{cases}
\omega P + (0, \omega^2), & \text{in case 1 with $p \equiv 4 \psmod{9}$;} \\
\omega^2 P + (0, \omega^2), & \text{in case 2 with $p \equiv 4 \psmod{9}$;} \\
\omega P + (0, \omega), & \text{in case 1 with $p \equiv 7 \psmod{9}$;} \\
\omega^2 P + (0, \omega), & \text{in case 2 with $p \equiv 7 \psmod{9}$.}
\end{cases}
\end{equation}
Since the element $\sigma$ leaves $\sqrt[3]{3}$ invariant, and since the point $(0, \omega)$ is mapped to $(0, \omega)$ under the twisting isomorphism (\ref{e:twist3}) in the models (\ref{e:model}), we see that the same equations hold for the point $Q$ replacing $P$.

Finally, in case 1 for $p \equiv 4 \pmod{9}$ we calculate
\[ \sigma(R) = \sum_{\varsigma \in \Gal(H_{3p}\,|\,L)} \sigma(\varsigma(Q)) = 
\sum_{\varsigma} \varsigma( \sigma(Q))= \omega R+\frac{p-1}{3}(0, \omega^2) = \omega R + (0, \omega^2), \]
since $(0, \omega^2)$ is a $3$-torsion point fixed by $\Gal(H_{3p}\,|\,L)$ and $[H_{3p}:L]=(p-1)/3$.  The other three cases follow similarly.
\end{proof}

\subsection{Descent from $L$ to $K$}

Unfortunately, Proposition~\ref{p:sigmar} does not imply that $R$ is nontorsion since there are torsion points in $E_1(L)$ that 
satisfy these equations. Namely, the torsion point $T=(1,1)$ satisfies $\sigma(T)=T=\omega^2 T+(0,\omega^2)$, and similarly $T=(1,-2)$ satisfies $\sigma(T)=T=\omega T + (0,\omega^2)$.  

But we turn this to our advantage: in case 1 the point $Y \colonequals R-T$ for $T=(1,-2)$  satisfies 
$\sigma(Y) = \omega Y$; and so again by the cubic twist isomorphism \eqref{e:twist3}, we obtain a point $Z \in E_{p}(K)$.  In case 2, we take $T=(1,1)$, let $Y = R - T$,  and find $\sigma(Y)=\omega^2 Y$ yielding $Z\in E_{p^2}(K)$.

\section{Nontorsion} \label{sec:nontors}

To prove that the point $R \in E_1(K(\sqrt[3]{p}))$ in \eqref{eqn:pointR} is nontorsion, and accordingly its twist $Z \in E_{p^i}(K)$ ($i = 1$ or $2$), we now consider its reduction modulo $p$.  

\subsection{Manipulation of $\eta$ product} \label{s:eta}

Recall Proposition \ref{modularparam} giving the modular parametrization $\Phi:X_0(243) \to E_9:y^2+y=x^3-1$, where 
\begin{equation} \label{eqn:xzdfe}
x(z)=\frac{\eta(9z)\eta(27z)}{\eta(3z)\eta(81z)}.
\end{equation}
In \eqref{e:pdef}  we considered the points $\tau=M(\omega p/9)$ for 
$M=\begin{pmatrix} 2 & -1 \\ 9 & -4 \end{pmatrix},\begin{pmatrix} 1 & 0 \\ -9 & 1 \end{pmatrix}$ in the two cases (Lemma \ref{lem:findmat}).  We now write the value of $x(\tau)$ in the form $f(p \tau_0)$ where $f$ is a modular function and $\tau_0$ does not depend on $p$.  

The function  
\begin{equation}  \label{eqn:fz273}
f(z) \colonequals \displaystyle{\frac{\eta(27z)}{\eta(3z)}} 
\end{equation}
is a modular unit on $\Gamma_0(81)$ by Ligozat's criterion.  

\begin{lem} \label{l:j}
Let $\pdot \in \Z$ satisfy $\pdot p \equiv 4,1 \psmod{27}$  in case \textup{1} or case \textup{2}, respectively.  Then
\begin{equation} \label{e:xz}
 x(\tau) = e^{\pi i/6} \sqrt{3} \, \frac{f(p(\omega-\pdot)/27)f(p\omega/9)}{f(p(\omega-\pdot)/9)}
\end{equation}
where $f$ is defined in \textup{\eqref{eqn:fz273}}.
\end{lem}

\begin{proof}
 We show the calculation for case 2; case 1 is similar.  With $M=\begin{pmatrix} 1 & 0 \\ -9 & 1 \end{pmatrix}$ and all $z \in \frakH$,
\begin{equation}
 81M(z) = \frac{9\omega p}{-\omega p+1} = \frac{9}{-\omega p+1}-9 = (T^{-9}S)((\omega p-1)/9)
 \end{equation}
where $S=\begin{pmatrix} 0 & -1 \\ 1 & 0 \end{pmatrix}$ and $T=\begin{pmatrix} 1 & 1 \\ 0 & 1 \end{pmatrix}$. Similarly
 \begin{equation}
 \begin{aligned}
 27M(z) &= (T^{-3}S)((\omega p-1)/3), \\
 9M(z) &=  (STS)(\omega p), \\
 3M(z) &=  (ST^3S)(\omega p/3).
 \end{aligned}
 \end{equation}
 By the transformation formulas for the Dedekind $\eta$-function
\begin{equation} 
\eta(T(z))=\eta(z+1)=e^{\pi i/12}\eta(z), \quad \eta(S(z))=\eta(-1/z) = \sqrt{-iz}\, \eta(z),
\end{equation}
we calculate:
\begin{equation} \label{eqn:812793M}
\begin{aligned}
 \eta(81 \tau) &= e^{\pi i/4}\sqrt{-i(\omega p-1)/9}\, \eta((\omega p-1)/9) \\
 \eta(27 \tau) &= e^{-\pi i/4}\sqrt{-i(\omega p-1)/3}\, \eta((\omega p-1)/3) \\
 \eta(9 \tau) &= e^{-\pi i/6}\sqrt{-i(\omega p-1)}\, \eta(\omega p) \\
 \eta(3 \tau) &= \sqrt{-i(\omega p-1)}\, \eta(\omega p/3).
\end{aligned}
\end{equation}
Plugging \eqref{eqn:812793M} into $x(\tau)$ as in \eqref{eqn:xzdfe} and rewriting slightly gives
\begin{equation} \label{eqn:xtauM91}
x(\tau)=x(M(\omega p/9)) = e^{\pi i/3} \sqrt{3} \cdot \frac{\eta(\omega p)}{\eta(\omega p-1)/9)} \cdot \frac{\eta((\omega p-1)/3)}{\eta(3\omega p)} \cdot \frac{\eta(3\omega p)}{\eta(\omega p/3)}.
 \end{equation}

Then with $\pdot=7,4$ as $p=4,7\psmod{9}$, let $k \colonequals (1-\pdot p)/9 \in 3\Z$.  
\begin{equation} \label{eqn:omegap19}
\frac{\omega p-1}{9}=\frac{p(\omega-\pdot)}{9}-k. 
\end{equation}
Using the transformation formula and \eqref{eqn:omegap19} gives:
\begin{equation}  
\frac{\eta(\omega p)}{\eta((\omega p-1)/9)} = \frac{e^{\pi i (jp)/12}\,\eta(p(\omega - \pdot))}{e^{\pi i(-k)/12}\,\eta(p(\omega -\pdot)/9)} = e^{\pi i/12} f(p(\omega-\pdot)/27)
\end{equation}
since $k+jp=1-8k \equiv 1 \psmod{24}$.  Similarly,
\begin{equation}
\frac{\eta((\omega p-1)/3)}{\eta(3\omega p)} = e^{-\pi i/4} \frac{1}{f(p(\omega-\pdot)/9)}.
\end{equation}
Plugging these into \eqref{eqn:xtauM91}, we obtain (\ref{e:xz}).
\end{proof}

\subsection{Reduction of $R$ modulo $p$}

We will use the tidy expression \eqref{e:xz} to reduce our mock Heegner points modulo $p$.  
The key result that allows  this is the following proposition.

\begin{prop} \label{shim2}   
Let $f(z)=\sum_n a_n q^n$ 
be a nonconstant modular function on $\Gamma_0(N)$ with $a_n \in \Z$ such that $f$ only has poles at cusps.  
Let $K$ be a quadratic imaginary field and $p$ a prime that splits in $K$.
Let $\tau \in \frakH$ have image in $X_0(Np)$ corresponding to a cyclic $Np$-isogeny $\varphi\colon E_1 \to E_2$ of elliptic curves with CM by orders in $K$.   Suppose that
the index $[\Z_K: \End(E_1)]$ is not divisible by $p$ but that $[\Z_K : \End(E_2)]$ is divisible by $p$.

  Let $H$ be the ring class field of $K$ associated to $\End(\varphi)$ and let $\Z_{H,(p)}$ denote the ring of $p$-integral elements of $H$.  Then 
$f(\tau), f(p \tau) \in H^\times$  are integral at each prime of $H$ above $p$ and satisfy the congruence
\begin{equation} 
f(\tau) \equiv f(p\tau)^p \pmod{p\Z_{H,(p)}}.
\end{equation} 
\end{prop}

Proposition~\ref{shim2} is proved in the appendix.  Using the proposition, we now finish the proof of our main result (Theorem~\ref{t:main}) by showing that $R$ is not torsion when 3 is not a cube modulo $p$.  We describe the case $j = 7$ (see Lemma~\ref{l:j}), the argument for the other cases only differing by constant factors (specifically, an explicit root of unity only depending on $j$).  We continue to use the model 
$y^2+y=3x^3-1$ for the curve $E_1$.
Recall from \eqref{r:xq} and Lemma \ref{l:j} that for the point  $Q \in E_1(H_{3p})$
we have 
\[ x(Q)= \frac{x(P)}{\sqrt[3]{3}} =  \frac{x(\tau)}{\sqrt[3]{3}} = e^{\pi i/6} \sqrt[6]{3} \, \frac{f(p(\omega-7)/27) f(p\omega/9)}{f(p(\omega-7)/9)}.\]
It is elementary to see that the assumptions of Proposition~\ref{shim2} are satisfied by $f$ and the points $\tau = \omega/9, (\omega - 7)/27, (\omega - 7)/9$.
The proposition therefore implies that
\begin{equation}
 x(Q)^p \equiv (e^{\pi i/6} \sqrt[6]{3} )^p \frac{f((\omega-7)/27) f(\omega/9)}{f((\omega-7)/9)} \pmod{p \overline{\Z}}. \label{e:evalf} 
 \end{equation}

We can evaluate the constant on the right in (\ref{e:evalf}) explicitly.

\begin{lem} \label{lem:compf27}
We have
\[ 
  \frac{f((\omega-7)/27)f(\omega/9)}{f((\omega-7)/9)} = -e^{\pi i/6}\frac{1}{\sqrt[6]{3}}. 
\]
\end{lem}

\begin{proof}
The function $h(z) := f(z/3)^3$ is  $\Gamma_0(9)$-invariant by Ligozat's criterion.
The point on $X_0(9)$ associated to $\omega/3 \in \frakH$ corresponds to the normalized isogeny $\la \omega/3 \ra \to \la 3\omega \ra$ of conductor $3$.  By the theory of modular units, $h(\omega/3)$ is a $3$-unit in the ring class field $H_3 = K$, and hence is equal to a unit in $\Z_K^\times$ times a power of $\sqrt{-3}$.  Numerically, we find that $h(\omega/3)=3\sqrt{-3}$ to several hundred digits of accuracy, so this must be an equality. We then calculate that 
\begin{equation} 
f(\omega/9)=e^{-\pi i/6}/\sqrt{3}.
\end{equation}
In a similar way, we compute
\[ f((\omega-7)/9)=-\omega^2/\sqrt[3]{9}, \quad f((\omega-7)/27)=-\omega/\sqrt[3]{3}, \]
and putting these together gives the result.
\end{proof}

Combining  \eqref{e:evalf} and Lemma~\ref{lem:compf27}, we obtain
\begin{equation}  \label{eqn:xtaupi6}
x(Q)^p \equiv (e^{\pi i/6} \sqrt[6]{3} )^p \left(\frac{-e^{\pi i/6}}{\sqrt[6]{3}}\right) = \omega^2 \left(-3\right)^{(p-1)/6} \pmod{p\overline{\Z}}. 
\end{equation}

Since $p \equiv 1 \pmod{3}$, we have $p\Z_L = (\frakp\overline{\frakp})^3$ where $\Z_L$ is the ring of integers of $L = K(\sqrt[3]{p})$ and each of $\frakp,\overline{\frakp}$ have residue field $\F_p$.  We  consider the pair
\begin{equation} 
(R\bmod{\frakp}, R\bmod\overline{\frakp}) \in E_1(\F_p)^2. 
\end{equation}

\begin{prop} \label{prop:3notpmainprop}
If $3$ is not a cube modulo $p$, then the image of $R \in E_1(L)$ in
$E_1(\F_p)^2$ is not equal to the image of any torsion point in $E_1(L)$, and hence $R$ is nontorsion.
\end{prop}

Before proving this proposition, we need one final lemma.

\begin{lem} \label{l:torsion} In the coordinates $y^2 + y = 3x^3 - 1$ for $E_1$, we have
\[E_1(L)_{\textup{tors}} = E_1(K)_{\textup{tors}}  = E_1[3] = \{ \infty, (0, \omega), (0, \omega^2),  (\omega^i,1), (\omega^i, -2): i = 0, 1, 2 \} \simeq (\Z/3\Z)^2. \]
\end{lem}

\begin{proof}
The curve $E_1$ has simplified Weierstrass model $y^2=x^3-432$ with $432=2^4 3^3$; since $\sqrt[3]{2} \not \in L$, we have $E_1[2](L)=\{\infty\}$.  The curve $E_1$ has good (supersingular) reduction at $2$.  The prime $2$ is unramified in the $S_3$-extension $L \supseteq \Q$; it is inert in $\Z_K$ and splits into three primes in $\Z_L$ with residue field of size $4$, and $\#E(\F_4)=9$.  By the injection of torsion \cite[(VII.3.2)]{silverman}, we conclude that $\#E_1(L) \mid 9$.  The 3-torsion points of $E_1$ listed explicitly in the proposition are clearly defined over $K \subset L$, completing the proof.
\end{proof}

\begin{proof}[Proof of Proposition~\textup{\ref{prop:3notpmainprop}}]   From \eqref{eqn:xtaupi6}  we have that $ x(Q)^p \equiv \omega^2 \left(-3\right)^{(p-1)/6}  \pmod{p\Z_{H_{3p}}}.$
Since $-3 \in \F_p^{\times 2}$, it follows that $(-3)^{(p-1)/6}$ is a cube root of unity in $\F_p^\times$; furthermore, this root of unity is trivial if and only if $3$ is a cube modulo $p$.
Meanwhile the image of $\omega^2$ in \[ \Z_K/p\Z_K \simeq \Z_K/{\frakp_K} \times \Z_K/\overline{\frakp}_K \simeq \F_p \times \F_p \] has the form $(u, u^2)$ where $1 \le u \le p-1$ is a primitive cube root of unity in $\F_p^\times = (\Z/p\Z)^\times$.  Therefore, (\ref{eqn:xtaupi6}) implies that the image of $x(Q)^p$ in $\F_p \times \F_p$ has the form \begin{equation} \label{e:xqred} \begin{cases}
(u, u^2), & \text{ if } 3 \text{ is a cube mod } p; \\
(u, 1) \text{ or } (1, u), & \text{ if } 3 \text{ is not a cube mod } p.
\end{cases}
\end{equation}  Of course, the same is therefore true for $x(Q)$.
In particular,  the image of $Q$ in each copy of $E_1(\F_p)$ is a 3-torsion point (namely one of the points $(u^i, 1)$ or $(u^i, -2)$ for $i = 0, 1, 2$).
Now
\begin{equation} \label{e:rcong}
 R=\Tr_{H_{3p}/L} Q \equiv \frac{p-1}{3}Q \equiv  \pm Q \text{ in } E_1(\F_p)^2,
 \end{equation}
with the sign $\pm$ according to  the cases $p \equiv 4, 7 \pmod{9}$.
The first congruence in (\ref{e:rcong}) follows since  $p$ is totally ramified in the extension $H_{3p}/L$.  
To prove the proposition, it therefore suffices to prove that the image of $Q$ in
$E_1(\F_p)^2$ is not equal to the image of a torsion point in $E_1(L)$ if  $3$ is not a cube modulo $p$.
However, this is easily checked directly using Lemma~\ref{l:torsion} and (\ref{e:xqred}).  For the nonzero torsion points  $T \in E_1[3]$, the images of $x(T)$ in $\F_p \times \F_p$ have the shape $(0, 0), (1, 1),$ or $(u, u^2)$ with $u$ a primitive cube root of unity in $\F_p^\times$, never  $(u, 1)$ or $(1, u)$.
\end{proof}

Of course, if $R$ is nontorsion, then the points $Y = R - T \in E_1(L)$ and $Z \in E_{p^i}(K)$ will be nontorsion as well.  Finally, since $E$ has CM by $\Z_K$ we have $\rk_{\Z}(E(\Q)) = \rk_{\Z_K}(E(K))$. Explicitly, if $Z \in E_{p^i}(K)$ is nontorsion then either $Z + \overline{Z}$ or $(\sqrt{-3}Z) + (\overline{\sqrt{-3} Z})$ will be a nontorsion point in $E_{p^i}(\Q)$.
This concludes the proof of Theorem~\ref{t:main}.

\subsection{Tables} \label{sec:tables}

In the following tables, we show the points constructed with our method, suggesting they are nontorsion whenever the corresponding twisted $L$-value is nonzero (see \S1.4).  We define
\[ L_{\textup{alg}}(E_{n},1) \colonequals L(E_n,1) \frac{2 \pi \sqrt[3]{n}}{ \sqrt{3} \Gamma(1/3)^3}, \]
the conjectural order of the Shafarevich--Tate group of $E_n$. We let $m(P)$ denote the index of $\langle P \rangle$ in the Mordell--Weil group $E(\Q)$.

\begin{small}
\[\def\arraystretch{1.3}\def\arraycolsep{0.5ex}
\begin{array}{c|cc|cc}
p & L_{\textup{alg}}(E_{3p^2},1) & (3\,|\,p)_3=1? & P \in E_p(\Q) & m(P) \\
\hline
7 & 1 & \textsf{no} & (2,-1) & 1 \\
13 & 4 & \textsf{no} & (\frac{2513}{1005},-\frac{1388}{1005}) & 2 \\
31 & 4 & \textsf{no} & (\frac{277028111}{119531076},\frac{316425265}{119531076}) & 2 \\
43 & 1 & \textsf{no} & (\frac{805}{228},-\frac{229}{228}) & 2 \\
61 & 0 & \textsf{yes} & \infty & - \\
67 & 9 & \textsf{yes} & (\frac{-3481613117857223908773469049678633}{610868942776961094346380627914232},\frac{3859176073959095744240009217935657}{610868942776961094346380627914232}) & 3 \\
79 & 1 & \textsf{no} & (\frac{26897}{6783},\frac{17320}{6783}) & 2 \\
97 & 4 & \textsf{no} & (-\frac{2799894968113535105}{200421477873478047},\frac{2832713504497390136}{200421477873478047}) & 4 \\
103 & 9 & \textsf{yes} & (\frac{846452740978167916651651}{2613111768231818449540464}, \frac{12247739733626179769224061}{2613111768231818449540464}) & 3 \\
139 & 4 & \textsf{no} & (\frac{54560}{13317},\frac{54943}{13317}) & 2 \\
151 & 9 & \textsf{yes} &  
(-\frac{123623\cdots 7041}{313952 \cdots 2740},\frac{1672043 \cdots 5041}{313952 \cdots 2740}), \hht(P) \approx 140.03 & 6 \\
157 & 4 & \textsf{no} & (-\frac{149538978691379960828806099105}{17911115779648062701697963576}, \frac{161931070975357602816944210593}{17911115779648062701697963576}) & 2 \\
193 & 0 & \textsf{yes} & \infty & -  
\end{array}\]

\[\def\arraystretch{1.3}\def\arraycolsep{0.5ex}
\begin{array}{c|cc|cc}
p & L_{\textup{alg}}(E_{3p},1) & (3\,|\,p)_3=1? & P \in E_{p^2}(\Q) & m(P) \\
\hline
7 & 1 & \textsf{no} & (-\frac{2}{3},\frac{11}{3}) & 1 \\
13 & 1 & \textsf{no} & (\frac{1589}{285},-\frac{464}{285}) & 2 \\
31 & 1 & \textsf{no} & (\frac{12376607}{1219092},-\frac{5368415}{1219092}) & 2 \\
43 & 4 & \textsf{no} & (\frac{3884810234333940170434868735}{316639715249572968055283052},
\frac{413561995142793125324177473}{316639715249572968055283052}) & 2 \\
61 & 0 & \textsf{yes} & \infty & - \\
67 & 0 & \textsf{yes} & \infty & - \\
79 & 1 & \textsf{no} & (\frac{416502767358398513}{77680272383924217},\frac{1418322935604634846}{77680272383924217}) & 1 \\
97 & 1 & \textsf{no} & (\frac{76769228526793}{20893884519009},\frac{440320075625234}{20893884519009}) & 1 \\
103 & 0 & \textsf{yes} & \infty & - \\
139 & 4 & \textsf{no} & 
(\frac{273171 \cdots 7720}{644917 \cdots 4681},-\frac{247724 \cdots 7279}{644917 \cdots 4681}), \hht(P) \approx 232.48 & 4 \\
151 & 0 & \textsf{yes} & \infty & - \\
157 & 1 & \textsf{no} & (-\frac{338502049691004117840147474335}{18567552055567917366723961524},\frac{581442015167638901460155379551}{18567552055567917366723961524}) & 2 \\
193 & 0 & \textsf{yes} & \infty & - 
\end{array}\]
\end{small}

\appendix
\section{Application of mod $p$ geometry}

\numberwithin{equation}{section}

In this appendix we prove
Proposition~\ref{shim2}.  The proposition will be deduced as a special case of a more general underlying geometric principle.
Let $X$ be a proper flat curve over a discrete valuation ring $R$ with mixed characteristic $(0,p)$. Let $F=\Frac R$ denote the fraction field of $R$ and let $k$ be the residue field of $R$.  Suppose that $X_F$ is smooth and geometrically connected. Suppose further that $X_k$ is semistable with two irreducible components, each smooth and geometrically connected.  Let $D$ be an $R$-finite flat closed subscheme of $X$ whose special fiber lies in the smooth locus of the special fiber of $X$.

Let $f \in \scrO_{X_F}(U_F)$ for $U \colonequals X \smallsetminus D$.  Let $\infty \in D(R)$ be such that the image of $f$ in the $\infty_F$-adic completion $\Frac(\widehat{O}_{X_F, \infty_F})$ of $F(X)$ belongs to the polar localization along $\infty$ of the $\infty$-adic completion of $\scrO_X$. More concretely, if $q$ is a local generator along $\infty$ of its ideal sheaf in $\scrO_X$ then we are supposing that the natural map $$F(X) \rightarrow F(\!(q)\!) = F[\![q]\!][1/q]$$ carries $f$ into $R[\![q]\!][1/q]$.  We claim that the following general congruence holds.

\begin{prop}\label{shim} 
Suppose that $g \in \scrO_{X_F}(U_F)$ is such that its image in ${\rm{Frac}}(\widehat{O}_{X_F,\infty_F}) = F(\!(q)\!)$ belongs to $R[\![q]\!][1/q]$ and has reduction modulo $pR$ coinciding with the image of $f^p$. Then for any $u \in U(R)$ such that $u$ and $\infty$ reduce into the same connected component of the smooth locus of $X_k$, we have $f(u), g(u) \in R$ and $g(u) \equiv f(u)^p \pmod{pR}$.
\end{prop}

We first show how this proposition implies Proposition~\ref{shim2}.

\begin{proof}[Proof of Proposition~\textup{\ref{shim2}}]  We apply Proposition~\ref{shim} with $X = X_0(Np)$ over the localization $R = \Z_{H,(\frakp)}$ with $F = H$ and $\frakp$ a prime above $p$; we take $D$ to be the closed subscheme of cusps including the cusp $\infty$; and the modular function $f$ as in Proposition~\ref{shim2}.

We let
$g(z) = f(pz)$ and $u = W_p(\tau)$ for $\tau$ as in Proposition~\ref{shim2} with $W_p$ the Atkin--Lehner involution of $X_0(Np)$.  Since the $q$-expansion of $f$ has coefficients in $\Z$, the $q$-expansions of $f^p$ and $g$ are congruent modulo $p$.

The point on $X_0(Np)$ associated to $\tau$ corresponds to a cyclic $Np$-isogeny $\varphi\colon E_1 \longrightarrow E_2$,
and we are assuming that $m = [\Z_K:\End(E_1)]$ is relatively prime to $p$, but that $p \mid m_2 := [\Z_K: \End(E_2)]$.
As we explain below, these conditions ensure that $\tau$ has reduction in the connected component of the smooth locus of $X_k$ corresponding to \'etale $p$-level structure (i.e., the component distinct from the one into which $\infty$ reduces). Therefore $u$ and $\infty$ have reduction into the same component of the smooth locus of $X_k$. Granting that, since $g(u) = f(\tau)$ and $f(u) = f(p\tau)$ (due to the $\Gamma_0(N)$-invariance of $f$) we  then get from  Proposition \ref{shim} that that $f(\tau)$ and $f(p\tau)$ belong to $R$ and satisfy $f(\tau) \equiv f(p\tau)^p \pmod {pR}$.

To see that $\tau$ has reduction with \'etale $p$-level structure, it is equivalent to show that its reduction does not have multiplicative $p$-level structure.  Suppose for the sake of contradiction that this is the case (i.e., that the reduction of $\tau$ does have multiplicative $p$-level structure).
Extending $F$ a finite amount if necessary, the $F$-point $\tau$ of the coarse space $Y_0(Np)$ comes from a CM elliptic curve $E$ over $R$,  and $E[p]$ then has connected-\'etale sequence over $R$ which (by canonicity) is stable by the order $\Z_{K,m}$. Hence, passing to generic fibers,
the subgroup $J$ of order $p$ in $\ker \varphi$ must be stable by $\Z_{K,m}$.  But then $E_1/J$ would have endomorphisms by $\Z_{K, m}$, and hence $p$ would not divide $m_2$ (since $E_2$ is  a quotient of $E_1/J$ by a subgroup of size $N$, which is prime to $p$).  This contradiction to our assumptions implies that $\tau$ has reduction with \'etale $p$-level structure and concludes the proof.
\end{proof}

We conclude with a proof of Proposition \ref{shim}.

\begin{proof}[Proof of Proposition~\textup{\ref{shim}}]
Pick an affine open $V \subset U$ around the reduction $u_k$ of $u$ such that $V_k$ is contained in the common irreducible component that contains the reductions of $\infty$ and $u$, so $V$ is $R$-smooth with geometrically connected (hence geometrically integral) fibers and $u \in V(R)$.  Since an integrally closed noetherian domain (such as $R[V]$) is the intersection in its fraction field of its localizations at all height-1 primes, the only obstacle to $f|_{V_K} \in K[V_K]$ coming from $R[V]$ is that the order of $f$ at the generic point of $V_k$ may be negative.

Assuming this order is negative, say $-m$, if $\pi$ is a uniformizer of $R$ then $\pi^m f$ comes from $R[V]$ and has nonzero reduction modulo $\pi$. To rule this out, we observe (by some elementary unraveling of definitions) that the image in $k[\![q]\!][1/q]$ of the reduction of $\pi^m f$ is the reduction of $\pi^m$ times the element of $R[\![q]\!][1/q]$ that is assumed to be the image of $f$ in $K(\!(q)\!) = K[\![q]\![1/q]$, and the latter reduction is clearly 0.  This is a contradiction.  The same reasoning applies to $g$, as well as to $(f^p - g)/p$, so it follows that 
$$f, g, \frac{f^p - g}{p} \in R[V].$$
Now evaluating at $u \in V(R)$ gives the desired conclusions concerning $f(u)$ and $g(u)$.
\end{proof}


\begin{thebibliography}{999}

\bibitem{CST15}
Li Cai, Jie Shu, and Ye Tian, \emph{Cube sum problem and an explicit Gross-Zagier formula}, accepted by Algebra \& Number Theory, 2015.

\bibitem{Darmon}
Henri Darmon, \emph{Rational points on modular elliptic curves}, CBMS Reg.\ Conf.\ Ser.\ in Math., vol.~101, Amer.\ Math.\ Soc., Providence, 2004. 

\bibitem{DV}
Samit Dasgupta and John Voight, \emph{Heegner points and Sylvester's conjecture}, Arithmetic geometry, Clay Math.\ Proc., vol.\ 8, Amer.\ Math.\ Soc., Providence, 2009, 91--102.

\bibitem{Elkies}
Noam Elkies, \emph{Heegner point computations}, Algorithmic number theory (ANTS-I, Ithaca, NY, 1994), Lecture Notes in Comp.\ Sci., vol.~877, 1994, 122--133.

\bibitem{ElkiesWWW}
Noam Elkies, Tables of fundamental integer solutions $(x,y,z)$ of $x^3+y^3=pz^3$ with $p$ a prime congruent to $4$ mod $9$ and less than $5000$ or congruent to $7$ mod $9$ and less than $1000$, available at \texttt{http://www.math.harvard.edu/\~{}elkies/sel\_p.html}.

\bibitem{GZ}
Benedict H.~Gross and Don Zagier, \emph{Heegner points and derivatives of $L$-series}, Inv.\ Math.\ \textbf{84} (1986), no.~2, 225--320.

\bibitem{kilford}
Lloyd Kilford, {\em Generating spaces of modular forms with $\eta$-quotients}, 
JP J.~Algebra Number Theory Appl.\ \textbf{8} (2007), no. 2, 213--226. 

\bibitem{Monsky}
Paul Monsky, \emph{Mock Heegner points and congruent numbers}, Math.~Z.\ \textbf{204} (1990), 45--68.

\bibitem{Ogg}
A.P.~Ogg, \emph{Modular functions}, Santa Cruz Conference on Finite Groups (Univ.
California, Santa Cruz, Calif., 1979), Proc.\ Sympos.\ Pure Math., vol.~37, Amer.\ Math.\ Soc., Providence, 1980, 521--532.

\bibitem{RVZ}
Fernando Rodriguez-Villegas and Don Zagier, \emph{Which primes are sums of two cubes?}, Number theory (Halifax, NS, 1994), CMS Conf.\ Proc., vol.\ 15, Amer.\ Math.\ Soc., Providence, 1995, 295--306.

\bibitem{Satge}
Philippe Satg\'e, \emph{Groupes de Selmer et corps cubiques}, J.~Number Theory\ \textbf{23} (1986), 294--317.

\bibitem{Satge2}
Philippe Satg\'e, \emph{Un analogue du calcul de Heegner}, Inv.~Math.\ \textbf{87} (1987), 425--439.

\bibitem{Selmer}
E.S.~Selmer, \emph{The diophantine equation $ax^3+by^3+cz^3=0$}, Acta Math.\ \textbf{85} (1951), 203--362.

\bibitem{ShuYin}
Jie Shu and Hongbo Yin, \emph{On the Sylvester conjecture}, preprint, 2017.

\bibitem{silverman} Joseph Silverman, {\em The arithmetic of elliptic curves.}
Second edition. Graduate Texts in Mathematics, \textbf{106}. Springer, Dordrecht, 2009. 513 pp.

\bibitem{Sylvester}
J.J.~Sylvester, \emph{On certain ternary cubic-form equations}, Amer.~J.~Math.\ \textbf{2} (1879), no.~4, 357--393.

\end{thebibliography}
\end{document}